\documentclass{amsart}
\usepackage{amsmath, amsfonts,amssymb,amsthm}
\usepackage{graphicx}
\usepackage{thmtools,thm-restate}
\usepackage{enumitem}
\usepackage{float}
\usepackage{tikz}
\usetikzlibrary{positioning,decorations.pathmorphing,arrows.meta}
\tikzset{main node/.style={circle,fill=black,draw,minimum width=4pt,inner sep=0pt}}
\newtheorem{theorem}{Theorem}[section]
\newtheorem{lemma}[theorem]{Lemma}
\newtheorem{cor}[theorem]{Corollary}

\theoremstyle{definition}
\newtheorem{definition}[theorem]{Definition}
\newtheorem{example}[theorem]{Example}

\newenvironment{claim}[2]{\vspace{0.1cm}\par\noindent\emph{Claim $#1$.}\space#2}{}

\newcommand{\N}{\mathbb{N}}
\newcommand{\K}[1]{\vec{K}_{#1} }

\newcommand{\U}{\mathcal{U}}

\begin{document}
\author{Carl B\"urger}
\author{Louis DeBiasio}\thanks{The research of the second author is supported in part by Simons Foundation Collaboration Grant \# 283194.}
\author{Hannah Guggiari}
\author{Max Pitz}

\address{(B\"urger) University of Hamburg, Department of Mathematics, Bundesstra{\ss}e 55 (Geomatikum), 20146 Hamburg, Germany}
\email{carl.buerger@uni-hamburg.de}

\address{(DeBiasio) Miami University, Department of Mathematics, Oxford, OH, 45056, United States}
\email{debiasld@miamioh.edu}

\address{(Guggiari) University of Oxford, Mathematical Institute, Oxford, OX2 6GG, United Kingdom}
\email{guggiari@maths.ox.ac.uk}

\address{(Pitz) University of Hamburg, Department of Mathematics, Bundesstra{\ss}e 55 (Geomatikum), 20146 Hamburg, Germany}
\email{max.pitz@uni-hamburg.de}

\title[Partitioning symmetric digraphs into monochromatic graphs]{Partitioning edge-coloured complete symmetric digraphs into monochromatic complete subgraphs}  


\begin{abstract}
Let $\K{\N}$ be the complete symmetric digraph on the positive integers. Answering a question of DeBiasio and McKenney \cite{DM16}, we construct a $2$-colouring of the edges of $\K{\N}$ in which every monochromatic path has density~$0$. 

However, if we restrict the length of monochromatic paths in one colour, then no example as above can exist: We show that every $(r+1)$-edge-coloured complete symmetric digraph (of arbitrary infinite cardinality) containing no directed paths of edge-length $\ell_i$ for any colour $i\leq r$ can be covered by $\prod_{i\leq r} \ell_i$ pairwise disjoint monochromatic complete symmetric digraphs in colour $r+1$. 

Furthermore, we present a stability version for the countable case of the latter result: We prove that the edge-colouring is uniquely determined on a large subgraph, as soon as the upper density of monochromatic paths in colour $r+1$ is bounded by $\prod_{i\in [r]}\frac{1}{\ell_i}$.
\end{abstract}

\maketitle 

\section{Introduction}
Let $K_{\mathbb{N}}$ be the complete graph on the positive integers and $\vec{K}_{\mathbb{N}}$ be the complete symmetric digraph on the positive integers. The \textit{upper density} of a set $A\subseteq\mathbb{N}$ is
\[
\bar{d}(A)=\limsup_{n\rightarrow\infty}\frac{|A\cap\{1,\dots,n\}|}n
\]
For a graph or digraph $G$ with vertex set $V(G)\subseteq\mathbb{N}$, we define the \textit{upper density} of $G$ to be that of $V(G)$.  Throughout this paper, by a $k$-colouring, we mean a $k$-edge-colouring. In a 2-colouring, we will assume that the colours are red and blue.  Given a directed graph $D$ and sets $A, B\subseteq V(D)$, we write $[A,B]$-edges to mean all edges $(x,y)\in E(D)$ with $x\in A$ and $y\in B$.  

For finite graphs, Gerencs\'er and Gy\'arf\'as \cite{GG} proved that in every 2-colouring of $K_n$ there is a monochromatic path on at least $(2n+1)/3$ vertices; furthermore, this is best possible.  Erd\H{o}s and Galvin \cite{EG93} proved an infinite analogue, showing that in every 2-colouring of $K_{\mathbb{N}}$ there exists a monochromatic infinite path with upper density at least $\frac23$; they also give a 2-colouring of $K_{\mathbb{N}}$ in which every monochromatic path has upper density at most $\frac89$.  Recently, DeBiasio and McKenney \cite{DM16} improved the lower bound, showing that, in every 2-colouring of $K_{\mathbb{N}}$, there exists a monochromatic infinite path with upper density at least $\frac34$.  For more colours, Rado \cite{R78} showed that, in every $r$-colouring of $K_{\mathbb{N}}$, there is a partition of the vertices into at most $r$ disjoint paths of distinct colours\footnote{In fact, Rado proved that in any $r$-colouring of the edges of $\vec{K}_{\mathbb{N}}$ there is a partition of the vertices into at most $r$ disjoint \emph{anti-directed} paths of distinct colours. This implies the undirected version stated above.}, which implies that there is a monochromatic path of upper density at least $1/r$. Elekes, Soukup, Soukup and Szentmikl\`{o}ssy \cite{ESSS17} have recently extended Rado's result for two colours to the complete graph on $\aleph_1$ where $\aleph_1$ is the smallest uncountable cardinal. Shortly after, Soukup \cite{S17} extended this further to any finite number of colours and complete graphs of any infinite cardinality.

For directed graphs, the picture is a little different. In the finite case, Raynaud \cite{R73} showed that, in any 2-colouring of $\vec{K}_n$, there is a directed Hamiltonian cycle $C$ with the following property: there are two vertices $a$ and $b$ such that the directed path from $a$ to $b$ along $C$ is red and the directed path from $b$ to $a$ along $C$ is blue.  As a corollary, in any 2-colouring of $\vec{K}_n$, there is a monochromatic directed path on at least $n/2+1$ vertices; this is easily seen to be best possible by partitioning the vertices of $\vec{K}_n$ into two sets $A, B$ with $||A|-|B||\leq 1$ and colouring the edge $(x,y)$ red if $x\in A$ and blue if $x\in B$.  
In this paper, we will be interested in the infinite directed case. In particular, we will be considering edge-colourings of $\vec{K}_{\mathbb{N}}$ and prove a variety of results relating to the upper density of paths in $\vec{K}_{\mathbb{N}}$.

Let $P=v_1v_2...$ be a path in $\vec{K}_{\mathbb{N}}$. We say that $P$ is a \textit{directed path} if every edge in $P$ is oriented in the same direction. By the \textit{length} of a path, we mean the number of edges in the path. DeBiasio and McKenney \cite{DM16} recently proved the following result.

\begin{theorem}
For every $\varepsilon>0$, there exists a 2-colouring of $\vec{K}_{\mathbb{N}}$ such that every monochromatic directed path has upper density less than $\varepsilon$.
\end{theorem}

DeBiasio and McKenney also asked the following natural question: does there exist a 2-colouring of $\vec{K}_{\mathbb{N}}$ in which every monochromatic directed path has upper density 0? In Section \ref{sec:density0}, we will give a positive answer to this question (taken from the manuscript \cite{G17} by the third author). We note that the same example was independently obtained by Jan Corsten \cite{C17}.

\begin{theorem}
\label{thm:density0}
There exists a 2-colouring of $\vec{K}_{\mathbb{N}}$ such that every monochromatic directed path has upper density 0.
\end{theorem}

In light of this result, it is natural to ask under what conditions we can guarantee the existence of a monochromatic path of positive density. It is easy to see (from Ramsey's Theorem) that every $r$-colouring of $\vec{K}_{\mathbb{N}}$ contains a monochromatic directed path of infinite length. The third author observed in \cite{G17} that if one restricts the maximal length of directed paths in the first colour, then there must be monochromatic paths in the second colour with non-vanishing upper density. More generally, the authors proved the following sequence of results:

\begin{itemize}
\item In the manuscript \cite{G17}, Guggiari shows that for any $(r+1)$-edge-colouring of $\vec{K}_{\mathbb{N}}$ for which there are no directed paths of length $\ell_i$ in colour $i$ for any $i\in [r]$, there is a monochromatic directed path in colour $r+1$ with upper density at least $\prod_{i\leq r} \frac{1}{\ell_i}$.
\item Confirming and extending a conjecture by Guggiari, B\"urger and Pitz showed in the manuscript \cite{BP17} that under the same assumptions as above, the vertex set $\N$ can be covered by $\prod_{i\leq r} \ell_i$ pairwise disjoint monochromatic directed paths in colour $r+1$. 
\end{itemize}

In Section~\ref{sec:rColors} of this paper, we will present a proof of the following statement, which nicely generalizes the  results above. Write $\vec{K}$ for the complete symmetric digraph on a finite or (not-necessarily-countable) infinite number of vertices.

\begin{theorem}\label{Louis Thm}
Let $c\colon E(\vec{K})\rightarrow [r+1]$ be an edge-colouring of $\vec{K}$ for which there is no directed path of length $\ell_i$ in colour $i$ for any $i \in [r]$. Then there is a partition of $V(\vec{K})$  into $\prod_{i\in[r]} \ell_i$ complete symmetric digraphs monochromatic in colour $r+1$.
\end{theorem}

In particular, for the countable directed complete digraph $\K{\N}$, we reobtain the result, now with a much shorter proof, that for any $(r+1)$-edge-colouring of $\vec{K}_{\mathbb{N}}$ for which there are no monochromatic directed paths of length $\ell_i$ in colour $i$ for any $i\in [r]$, the vertex set $\N$ can be covered by $\prod_{i\leq r} \ell_i$ pairwise disjoint monochromatic directed paths in colour $r+1$. 

It is not hard to see that this result is best possible: There are colourings of $\K{\N}$, which we will call \emph{cube colourings} for their geometric structure, that witness optimality of the above result, see Theorem~\ref{thm_cubecolouringworks} below. For example, it is not hard to see that any blue monochromatic path in the colouring of Figure~\ref{fig_cubecolouring} has upper density at most $\frac16$.

\begin{figure}[H]
\centering
\begin{tikzpicture}
\draw (-3,1) ellipse (1cm and 0.5cm) node {$U_{0,0}$};
\draw (0,1) ellipse (1cm and 0.5cm) node {$U_{0,1}$};
\draw (3,1) ellipse (1cm and 0.5cm) node {$U_{0,2}$};
\draw (-3,-1) ellipse (1cm and 0.5cm) node {$U_{1,0}$};
\draw (0,-1) ellipse (1cm and 0.5cm) node {$U_{1,1}$};
\draw (3,-1) ellipse (1cm and 0.5cm) node {$U_{1,2}$};

\draw[-{Latex[length=2mm,width=2mm]},green,thick] (-1,1) to (-2,1);
\draw[-{Latex[length=2mm,width=2mm]},green,thick] (2,1) to (1,1);
\draw[-{Latex[length=2mm,width=2mm]},green,thick] (2,1.3) to [bend right] (-2,1.3);

\draw[-{Latex[length=2mm,width=2mm]},green,thick] (-1,-1) to (-2,-1);
\draw[-{Latex[length=2mm,width=2mm]},green,thick] (2,-1) to (1,-1);
\draw[-{Latex[length=2mm,width=2mm]},green,thick] (2,-1.3) to [bend left] (-2,-1.3);

\draw[-{Latex[length=2mm,width=2mm]},red,thick] (-3,-0.5) to (-3,0.5);
\draw[-{Latex[length=2mm,width=2mm]},red,thick] (0,-0.5) to (0,0.5);
\draw[-{Latex[length=2mm,width=2mm]},red,thick] (3,-0.5) to (3,0.5);

\draw[-{Latex[length=2mm,width=2mm]},red,thick] (2.2,-0.7) to (0.9,0.8);
\draw[-{Latex[length=2mm,width=2mm]},red,thick] (-2.2,-0.7) to (-0.9,0.8);
\draw[-{Latex[length=2mm,width=2mm]},red,thick] (0.9,-0.8) to (2.2,0.7);
\draw[-{Latex[length=2mm,width=2mm]},red,thick] (-0.9,-0.8) to (-2.2,0.7);

\draw[-{Latex[length=2mm,width=2mm]},red,thick] (2.1,-0.8) to (-2.1,0.8);
\draw[-{Latex[length=2mm,width=2mm]},red,thick] (-2.1,-0.8) to (2.1,0.8);

\node[right,red] at (5,0.5) {1 = Red};
\node[right,green] at (5,0) {2 = Green};
\node[right,blue] at (5,-0.5) {3 = Blue};
\end{tikzpicture}
\caption{Cube colouring for $\ell_1=2$ and $\ell_2=3$. All edges not shown are blue. All partition classes $U_{x,y}$ have density $\frac16$.}
\label{fig_cubecolouring}
\end{figure}
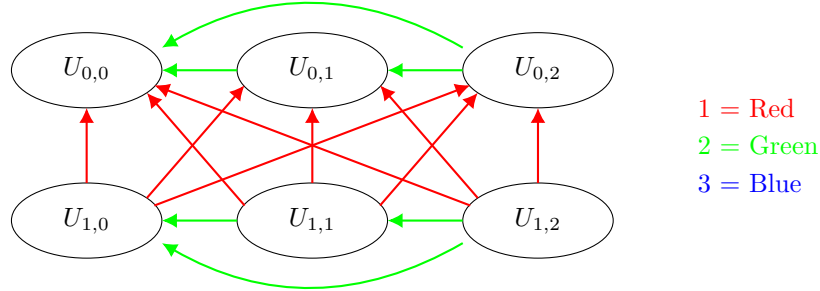

Extending the work of Guggiari in \cite[Theorem~1.4]{G17}, we will prove a stability version of Theorem~\ref{Louis Thm} in Section~\ref{sec:Stability} of this paper. This stability result, Theorem~\ref{Stability Thm}, says that any colouring $c\colon E(\K{\N})\rightarrow [r+1]$ for which there is no directed path of length $\ell_i$ in colour $i$ for any $i \in [r]$ and every directed path in colour $r+1$ has upper density at most $\prod_{i\in [r]}\frac{1}{\ell_i}$, must essentially be of the same cubic structure as the \emph{cube colouring}, where the meaning of `essentially' will be explained in Section~\ref{sec:Stability}.

\section{An example where all monochromatic paths have density $0$}
\label{sec:density0}
\begin{proof}[Proof of Theorem \ref{thm:density0}]
We colour the edges of $\vec{K}_{\mathbb{N}}$ in the following way. Let $m,n\in\mathbb{N}$ be distinct positive integers. Set $t=\min\{s\in\mathbb{N}:m\not\equiv n\mod2^s\}$. Exchanging $m$ and $n$ if necessary, we may assume that $m\equiv x\mod2^t$ where $x\in\{0,\dots,2^{t-1}-1\}$ and $n\equiv2^{t-1}+x\mod2^t$. We colour $(m,n)$ red and $(n,m)$ blue.

Let $P$ be any monochromatic directed path. If $P$ is a finite path, then $\bar{d}(P)=0$. Therefore, we may assume that $P$ is an infinite path. Without loss of generality, $P$ is red.
\\
\\
\emph{Inductive Hypothesis.} For any $k\in\mathbb{N}$, there exists $i\in\{0,\dots,2^k-1\}$ such that $P$ is eventually contained within the set $\{n\in\mathbb{N}:n\equiv i\mod2^k\}$. Hence, $\bar{d}(P)\leq2^{-k}$.
\\
\\
\emph{Base Case.} For $i\in\{0,1\}$, let $A_i=\{n\in\mathbb{N}:n\equiv i\mod2\}$. The sets $A_0$ and $A_1$ partition the vertices of $\vec{K}_{\mathbb{N}}$. Suppose $P$ contains a vertex $u\in A_1$. Then all of the vertices occurring after $u$ in $P$ must also be in $A_1$ because $P$ is a red directed path and hence $P$ is eventually contained within $A_1$. If $P$ does not contain a vertex from $A_1$, then $P$ must be completely contained within $A_0$. Hence, for some $i\in\{0,1\}$, we have $\bar{d}(P)\leq\bar{d}(A_i)=\frac12$.
\\
\\
\emph{Inductive Step.} Fix $k\geq2$ and partition the vertices of $\vec{K}_{\mathbb{N}}$ into $2^k$ sets $A_0,\dots,A_{2^k-1}$ based on their residue modulo $2^k$ (see Figure \ref{fig:densityzero}). By the inductive hypothesis, there exists $i\in\{0,\dots,2^{k-1}-1\}$ such that $P$ is eventually contained within the set $A_i\cup A_{2^{k-1}+i}$. By using the same argument as in the base case, we find that, if $P$ contains a vertex in $A_{2^{k-1}+i}$, then it is eventually contained within this set; otherwise, it is eventually contained within $A_i$. Hence, there exists $j\in\{0,\dots,2^k-1\}$ such that $P$ is eventually contained within $A_j$ and so $\bar{d}(P)\leq\bar{d}(A_j)=2^{-k}$.

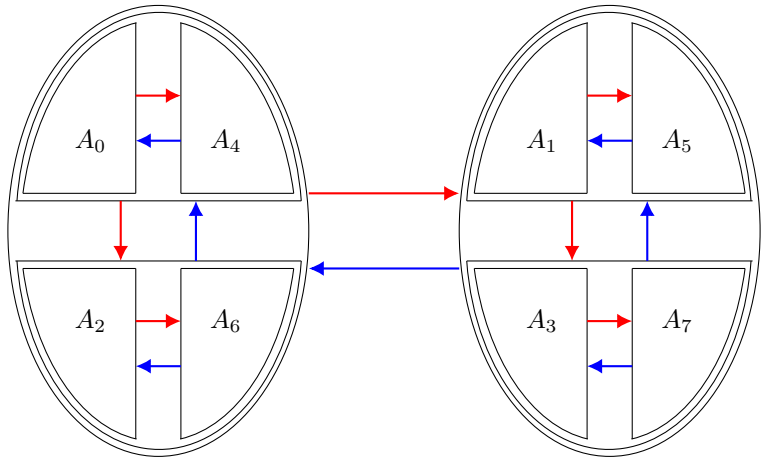
\begin{figure}[H]
\centering
\begin{tikzpicture}
\draw (-3,0) ellipse (2cm and 3cm) node {};
\draw (3,0) ellipse (2cm and 3cm) node {};
\draw[-{Latex[length=2mm,width=2mm]},blue,thick] (1,-0.5) to (-1,-0.5);
\draw[-{Latex[length=2mm,width=2mm]},red,thick] (-1,0.5) to (1,0.5);
\draw (-1.1,0.4) arc (6:174:1.9cm and 2.8cm);
\draw (-1.1,0.4) -- (-4.9,0.4);
\draw (-1.1,-0.4) arc (6:174:1.9cm and -2.8cm);
\draw (-1.1,-0.4) -- (-4.9,-0.4);
\draw[-{Latex[length=2mm,width=2mm]},blue,thick] (-2.5,-0.4) to (-2.5,0.4);
\draw[-{Latex[length=2mm,width=2mm]},red,thick] (-3.5,0.4) to (-3.5,-0.4);
\draw (1.1,0.4) arc (6:174:-1.9cm and 2.8cm);
\draw (1.1,0.4) -- (4.9,0.4);
\draw (1.1,-0.4) arc (6:174:-1.9cm and -2.8cm);
\draw (1.1,-0.4) -- (4.9,-0.4);
\draw[-{Latex[length=2mm,width=2mm]},blue,thick] (3.5,-0.4) to (3.5,0.4);
\draw[-{Latex[length=2mm,width=2mm]},red,thick] (2.5,0.4) to (2.5,-0.4);
\draw (1.2,0.5) arc (12:82:-1.8cm and 2.9cm);
\draw (1.2,0.5) -- (2.7,0.5);
\draw (2.7,2.77) -- (2.7,0.5);
\draw (4.8,0.5) arc (12:82:1.8cm and 2.9cm);
\draw (4.8,0.5) -- (3.3,0.5);
\draw (3.3,2.77) -- (3.3,0.5);
\draw[-{Latex[length=2mm,width=2mm]},blue,thick] (3.3,1.2) to (2.7,1.2);
\draw[-{Latex[length=2mm,width=2mm]},red,thick] (2.7,1.8) to (3.3,1.8);
\draw (-1.2,0.5) arc (12:82:1.8cm and 2.9cm);
\draw (-1.2,0.5) -- (-2.7,0.5);
\draw (-2.7,2.77) -- (-2.7,0.5);
\draw (-4.8,0.5) arc (12:82:-1.8cm and 2.9cm);
\draw (-4.8,0.5) -- (-3.3,0.5);
\draw (-3.3,2.77) -- (-3.3,0.5);
\draw[-{Latex[length=2mm,width=2mm]},blue,thick] (-2.7,1.2) to (-3.3,1.2);
\draw[-{Latex[length=2mm,width=2mm]},red,thick] (-3.3,1.8) to (-2.7,1.8);
\draw (1.2,-0.5) arc (12:82:-1.8cm and -2.9cm);
\draw (1.2,-0.5) -- (2.7,-0.5);
\draw (2.7,-2.77) -- (2.7,-0.5);
\draw (4.8,-0.5) arc (12:82:1.8cm and -2.9cm);
\draw (4.8,-0.5) -- (3.3,-0.5);
\draw (3.3,-2.77) -- (3.3,-0.5);
\draw[-{Latex[length=2mm,width=2mm]},blue,thick] (3.3,-1.8) to (2.7,-1.8);
\draw[-{Latex[length=2mm,width=2mm]},red,thick] (2.7,-1.2) to (3.3,-1.2);
\draw (-1.2,-0.5) arc (12:82:1.8cm and -2.9cm);
\draw (-1.2,-0.5) -- (-2.7,-0.5);
\draw (-2.7,-2.77) -- (-2.7,-0.5);
\draw (-4.8,-0.5) arc (12:82:-1.8cm and -2.9cm);
\draw (-4.8,-0.5) -- (-3.3,-0.5);
\draw (-3.3,-2.77) -- (-3.3,-0.5);
\draw[-{Latex[length=2mm,width=2mm]},blue,thick] (-2.7,-1.8) to (-3.3,-1.8);
\draw[-{Latex[length=2mm,width=2mm]},red,thick] (-3.3,-1.2) to (-2.7,-1.2);
\node at (-3.9,1.2) {$A_0$};
\node at (2.1,1.2) {$A_1$};
\node at (-3.9,-1.2) {$A_2$};
\node at (2.1,-1.2) {$A_3$};
\node at (-2.1,1.2) {$A_4$};
\node at (3.9,1.2) {$A_5$};
\node at (-2.1,-1.2) {$A_6$};
\node at (3.9,-1.2) {$A_7$};
\end{tikzpicture}
\caption{Diagram showing the edges between sets for $k=3$.}
\label{fig:densityzero}
\end{figure}
\noindent
The inductive hypothesis holds for every $k\in\mathbb{N}$. Therefore, if $P$ is any monochromatic directed path, we have that the upper density of $P$ is at most $2^{-k}$ for every $k\in\mathbb{N}$. Hence, $P$ has upper density 0.
\end{proof}

\section{Partitioning complete symmetric digraphs}\label{sec:rColors}
We will make use of the following well-known result.  We give a proof both for completeness and since many references only handle the finite case.

\begin{theorem}[Gallai \cite{G68}, Hasse \cite{H65}, Roy \cite{R73}, Vitaver \cite{V62}]\label{GHRV}
Let $D$ be a (not-necessarily-countable) directed graph and let $G$ be the underlying graph of $D$. If the longest path in $D$ has length $0\le k <\infty$, then $\chi(G)\le k+1$. 
\end{theorem}

\begin{proof}
Let $D'$ be a maximal acyclic subgraph of $D$; that is, take a well ordering of the edges of $D$ and add them one a time subject to the condition that a (finite) cycle is not created (or take the poset of acyclic subgraphs of $D$ ordered by inclusion and choose a maximal element by Zorn's lemma).  For each $0\leq i\leq k$, let $$U_i=\{v\in V(D): \text{ the length of the longest path in } D' \text{ starting at } v \text{ is } i\}.$$  By the hypothesis, $\{U_0, U_1, \dots, U_k\}$ partitions $V(D)$.  

We first claim that if $(y,x)\in E(D')$ with $x\in U_i$ and $y\in U_j$, then $i<j$. Let $P$ be a path in $D'$ of length $i$ which starts at $x$.  Since $D'$ is acyclic, $y\not\in V(P)$ and thus $yP$ is a path of length $i+1$ in $D'$ which, since $y\in U_j$,  implies $j\geq i+1$.  Next we claim that if $(x,y)\in E(D)\setminus E(D')$ with $x\in U_i$ and $y\in U_j$, then $i<j$.  Since $(x,y)\in E(D)\setminus E(D')$, the addition of $(x,y)$ to $D'$ must create a cycle, which implies that there is a $y-x$-path in $D'$.  By the previous claim, this implies that $i<j$.  Together, these two claims imply that there is no edge from $D$ with both endpoints in $U_i$ for any $0\leq i\leq k$ and thus $\{U_0, U_1, \dots, U_k\}$ is a proper colouring of $G$.  
\end{proof}

\begin{proof}[Proof of Theorem \ref{Louis Thm}]
For $r=0$, the result is trivial (note $\prod_{i=1}^r\ell_i=1$ in this case).  Let $r\geq 1$ and suppose the result holds for $r$-colourings satisfying the required path length condition.  Now consider an $(r+1)$-colouring in which every path of colour $1\leq i\leq r$ has length at most $\ell_i-1$.  Apply Theorem \ref{GHRV} to the digraph induced by the edges of colour $r$ to get a partition $\{U_0, \dots, U_{\ell_{r}-1}\}$ of $V$ such that each $U_j$ contains no edges of colour $r$. So each $U_j$ is an $r$-coloured complete symmetric digraph such that, for all $i\in [r-1]$, every path of colour $i$ has length at most $\ell_i-1$.  Thus by induction, there is a partition of each $U_j$ into 
$\prod_{i=1}^{r-1}\ell_i$ complete symmetric digraphs of colour $r+1$, giving a partition of $\vec{K}$ into a total of $\ell_r\prod_{i=1}^{r-1}\ell_i=\prod_{i=1}^{r}\ell_i$ complete symmetric digraphs of colour $r+1$.
\end{proof}

\begin{cor}[{cf. \cite{G17}}]\label{cor: exist. of dense mono. path}
Let $c\colon E(\K{\N})\rightarrow [r+1]$ be an edge-colouring of $\K{\N}$ for which there is no directed path of length $\ell_i$ in colour $i$ for any $i \in [r]$. Then there exists a directed path of colour $r+1$ with upper density at least $\prod_{i\in r} \frac{1}{\ell_i}$.
\end{cor}

\begin{cor}[{cf. \cite{BP17}}] \label{cor: exist. of dense mono. path}
Let $c\colon E(\K{\N})\rightarrow [r+1]$ be an edge-colouring of $\K{\N}$ for which there is no directed path of length $\ell_i$ in colour $i$ for any $i \in [r]$. Then the vertex set $\N$ can be partitioned into at most $\prod_{i\in r} \ell_i$ many monochromatic directed paths of colour $r+1$.
\end{cor}

This is best possible as shown by the \emph{cube colouring} on any \emph{cube partition} of $\N$ with equally upper-dense partition classes:

\begin{definition}[Cube partition]
For positive integers $\ell_1,\dots,\ell_r$ a partition $$\U=(U_i\colon i\in \prod_{i\in[r]} \{0,\dots ,\ell_i-1\})$$ of $\N$ indexed by the \emph{cube} $\prod_{i\in[r]} \{0,\dots ,\ell_i-1\}$ is called a \emph{cube partition} (of $\N$ of order $(\ell_1,\dots,\ell_r)$).
\end{definition}

\begin{definition}[Cube colouring]
For a cube partition $\U$ of $\N$, define the cube colouring $c_\U$ on $\K{\N}$ as follows: Consider an edge $(m,n)\in E(\K{\N})$. If $m,n\in U_{i_1,\dots,i_r}$, then colour both $(m,n)$ and $(n,m)$ with colour $r+1$. If not, suppose that $m\in U_{i_1,\dots,i_r}$ and $n\in U_{j_1,\dots,j_r}$. Let $k=\textup{min}\{k'\in [r]\colon i_{k'}\neq j_{k'}\}$ and $i_k<j_k$ (say). Colour $(m,n)$ with colour $r+1$ and $(n,m)$ with colour $k$. 
\end{definition}

See Figure~\ref{fig_cubecolouring} in the introduction for the case with three colours and $\ell_1=2$ and $\ell_2=3$.

\begin{theorem}
\label{thm_cubecolouringworks}
Let $c_\U$ be the cube colouring on a cube partition $\U$ with equally (upper-) dense partition classes. Then there is no directed path of length $\ell_i$ in colour $i$ for any $i\in [r]$ and every directed monochromatic path of colour $r+1$ has upper density at most $\prod_{i\in [r]}\frac{1}{\ell_i}$.  
\end{theorem}

\begin{proof}
Let $f$ map a vertex $v\in \K{\N}$ to the index of the partition class of $\U$ that contains $v$. For $i\in [r]$, let $f_i$ map a vertex $v$ to the $i$th entry of $f(v)$. 

First, consider a monochromatic forward directed path $P=v_0\dots v_m$ in some colour $i\in[r]$. Then $(f_i(v_0),\dots ,f_i(v_m))$ is a strictly decreasing sequence in $\{0,\dots, \ell_i-1\}$ and thus $P$ has length smaller than $\ell_i$. The proof for backward directed case is analogous. 

Second, consider a monochromatic forward directed path $P$ with colour $r+1$. If $P$ is finite, then it has upper density $0$ so we may assume that $P$ is infinite, $P=v_0v_1v_2\dots$ (say). The sequence $f(v_0)f(v_1)f(v_2)..$ is an increasing sequence of indices with respect to the lexicographic order on the cube $\prod_{i\in[r]}\{0,\dots,\ell_i-1\}$. Hence the vertices of $P$ are eventually contained in some partition class of $\U$. On the other hand, each partition class in $\U$ has upper density exactly $\prod_{i\in [r]}\frac{1}{\ell_i}$ completing the proof. The proof for the backward directed monochromatic paths of colour $r+1$ is analogue.   
\end{proof}

\section{Cube-like structures and stability theorem}\label{sec:Stability}

The main result of this section, our stability result, says that any optimal colouring $c\colon E(\K{\N})\rightarrow [r+1]$ for which there is no directed path of length $\ell_i$ in colour $i$ for any $i \in [r]$ and every directed path in colour $r+1$ has upper density at most $\prod_{i\in [r]}\frac{1}{\ell_i}$ must agree with our \emph{cube colouring} from above on the following spanning subgraph, which we call the \emph{slide digraph}.

\begin{definition}[Slide digraph] 
For a cube partition $\U$, the digraph $D_\U$ on $\N$ is defined as follows: A pair $(m,n)$ of distinct integers $m\in U_{i_1,\dots,i_r}$ and $n\in U_{j_1,\dots, j_r}$ is an edge of $D_{\U}$ if 
\begin{itemize}
\item there is an index $k$ such that $i_k>j_k$ and $i_{k'}=j_{k'}$ for all $k'\neq k$ (cf.~Figure~\ref{fig: slide ladder}), or 
\item $i_k\le j_k$ for every $k\in [r]$ (cf. Figure \ref{fig: the slide}). 
\end{itemize}
\end{definition}

\begin{figure}[H]
\centering
\begin{tikzpicture}
\draw (-3,1) ellipse (1cm and 0.5cm) node {$U_{0,0}$};
\draw (0,1) ellipse (1cm and 0.5cm) node {$U_{0,1}$};
\draw (3,1) ellipse (1cm and 0.5cm) node {$U_{0,2}$};
\draw (-3,-1) ellipse (1cm and 0.5cm) node {$U_{1,0}$};
\draw (0,-1) ellipse (1cm and 0.5cm) node {$U_{1,1}$};
\draw (3,-1) ellipse (1cm and 0.5cm) node {$U_{1,2}$};

\draw[-{Latex[length=2mm,width=2mm]},green,thick] (-1,1) to (-2,1);
\draw[-{Latex[length=2mm,width=2mm]},green,thick] (2,1) to (1,1);
\draw[-{Latex[length=2mm,width=2mm]},green,thick] (2,1.3) to [bend right] (-2,1.3);

\draw[-{Latex[length=2mm,width=2mm]},green,thick] (-1,-1) to (-2,-1);
\draw[-{Latex[length=2mm,width=2mm]},green,thick] (2,-1) to (1,-1);
\draw[-{Latex[length=2mm,width=2mm]},green,thick] (2,-1.3) to [bend left] (-2,-1.3);

\draw[-{Latex[length=2mm,width=2mm]},red,thick] (-3,-0.5) to (-3,0.5);
\draw[-{Latex[length=2mm,width=2mm]},red,thick] (0,-0.5) to (0,0.5);
\draw[-{Latex[length=2mm,width=2mm]},red,thick] (3,-0.5) to (3,0.5);

\node[right,red] at (5,0.5) {1 = Red};
\node[right,green] at (5,0) {2 = Green};
\end{tikzpicture}
\caption{First type of edges of the slide digraph for $\ell_1=2$ and $\ell_2=3$.}
\label{fig: slide ladder}
\end{figure}
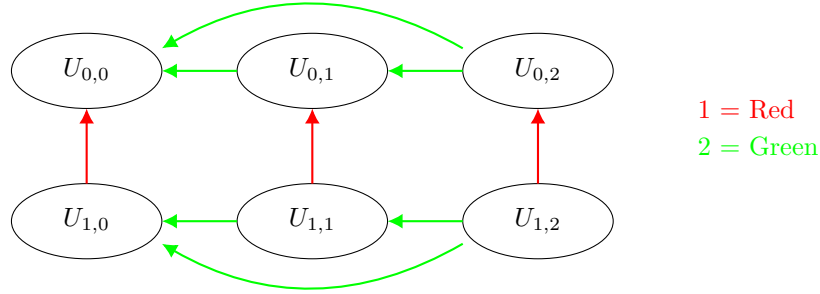

\begin{figure}[H]
\centering
\begin{tikzpicture}
\draw (-3,1) ellipse (1cm and 0.5cm) node {$U_{0,0}$};
\draw (0,1) ellipse (1cm and 0.5cm) node {$U_{0,1}$};
\draw (3,1) ellipse (1cm and 0.5cm) node {$U_{0,2}$};
\draw (-3,-1) ellipse (1cm and 0.5cm) node {$U_{1,0}$};
\draw (0,-1) ellipse (1cm and 0.5cm) node {$U_{1,1}$};
\draw (3,-1) ellipse (1cm and 0.5cm) node {$U_{1,2}$};

\draw[-{Latex[length=2mm,width=2mm]},blue,thick] (-2,1) to (-1,1);
\draw[-{Latex[length=2mm,width=2mm]},blue,thick] (1,1) to (2,1);
\draw[-{Latex[length=2mm,width=2mm]},blue,thick] (-2,1.3) to [bend left] (2,1.3);

\draw[-{Latex[length=2mm,width=2mm]},blue,thick] (-2,-1) to (-1,-1);
\draw[-{Latex[length=2mm,width=2mm]},blue,thick] (1,-1) to (2,-1);
\draw[-{Latex[length=2mm,width=2mm]},blue,thick] (-2,-1.3) to [bend right] (2,-1.3);

\draw[-{Latex[length=2mm,width=2mm]},blue,thick] (-3,0.5) to (-3,-0.5);
\draw[-{Latex[length=2mm,width=2mm]},blue,thick] (0,0.5) to (0,-0.5);
\draw[-{Latex[length=2mm,width=2mm]},blue,thick] (3,0.5) to (3,-0.5);

\draw[-{Latex[length=2mm,width=2mm]},blue,thick] (-3.6,0.7) to (-3.6,1.3);
\draw[-{Latex[length=2mm,width=2mm]},blue,thick] (-2.4,1.3) to (-2.4,0.7);
\draw[-{Latex[length=2mm,width=2mm]},blue,thick] (-2.4,-0.7) to (-2.4,-1.3);
\draw[-{Latex[length=2mm,width=2mm]},blue,thick] (-3.6,-1.3) to (-3.6,-0.7);

\draw[-{Latex[length=2mm,width=2mm]},blue,thick] (3.6,0.7) to (3.6,1.3);
\draw[-{Latex[length=2mm,width=2mm]},blue,thick] (2.4,1.3) to (2.4,0.7);
\draw[-{Latex[length=2mm,width=2mm]},blue,thick] (2.4,-0.7) to (2.4,-1.3);
\draw[-{Latex[length=2mm,width=2mm]},blue,thick] (3.6,-1.3) to (3.6,-0.7);

\draw[-{Latex[length=2mm,width=2mm]},blue,thick] (-0.6,0.7) to (-0.6,1.3);
\draw[-{Latex[length=2mm,width=2mm]},blue,thick] (0.6,1.3) to (0.6,0.7);
\draw[-{Latex[length=2mm,width=2mm]},blue,thick] (0.6,-0.7) to (0.6,-1.3);
\draw[-{Latex[length=2mm,width=2mm]},blue,thick] (-0.6,-1.3) to (-0.6,-0.7);

\draw[-{Latex[length=2mm,width=2mm]},blue,thick] (0.9,0.8) to (2.2,-0.7);
\draw[-{Latex[length=2mm,width=2mm]},blue,thick] (-2.2,0.7) to (-0.9,-0.8);

\draw[-{Latex[length=2mm,width=2mm]},blue,thick] (-2.1,0.8) to (2.1,-0.8);

\node[right,blue] at (5,-0.5) {3 = Blue};
\end{tikzpicture}
\caption{Second type of edges of the slide digraph for $\ell_1=2$ and $\ell_2=3$.}
\label{fig: the slide}
\end{figure}
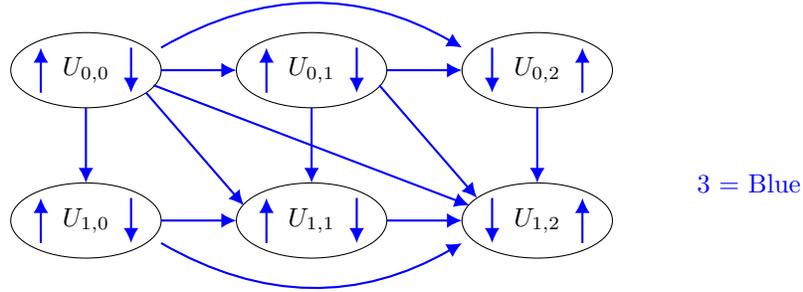

\begin{theorem}\label{Stability Thm}
Let $c\colon E(\K{\N})\rightarrow [r+1]$ be an edge-colouring of $\K{\N}$ for which there is no directed path of length $\ell_i$ in colour $i$ for any $i \in [r]$. Moreover, assume that every directed path in colour $r+1$ has upper density at most $\prod_{i\in [r]}\frac{1}{\ell_i}$. Then there exists a cube partition $\U$ of order $(\ell_1,\dots,\ell_r)$ of equally (upper-) dense partition classes and a finite set $F\subseteq \N$ such that $c_\U=c$ on $D_\U- F$. 
\end{theorem}

For the proof, we first need the following lemma.

\begin{lemma}\label{Transitive Tournament Lemma}
For any $n \in \N$, a spanning subgraph $G\subseteq \K{n}$ contains a spanning transitive tournament if and only if $\K{n} - E(G) $ is acyclic. 
\end{lemma}

\begin{proof}
If the spanning subgraph $G$ has a spanning transitive tournament $T$, then $\K{n}-E(G)$ is a subgraph of the transitive tournament $\K{n}-E(T)$ and hence acyclic. 

The converse is proved via induction on $n$. The base case is clear. Assume that the statement is true for $n-1$ and consider a spanning subgraph $G\subseteq \K{n}$ such that $\K{n}-E(G)$ is acyclic. Recall that every finite directed acyclic graph has a \emph{source}, i.e.\ a vertex of in-degree $0$. Fix a source $v$ in $\K{n}-G$ and claim that $v$ has in-degree $n-1$ in $G$: Indeed, since $v$ has in-degree $n-1$ in $\K{n}$ and is a source of $\K{n}-E(G)$, it follows that all these incoming edges must be contained in the subgraph $G$. By induction, the graph $G-v$ contains a transitive tournament $T$ spanning $G-v$. Then $E(T)\cup \{(w,v)\colon w\in T\}$ is the desired transitive spanning tournament contained in $G$. 
\end{proof}

\begin{proof}[Proof of Theorem~\ref{Stability Thm}]
Apply Theorem \ref{Louis Thm} to obtain a partition $\mathcal{U}$ of $\N$ into $\prod_{i\in [r]}\ell_i$ complete symmetric digraphs in colour $r+1$. Up to deleting finitely many vertices, these partition classes are the partition classes of the slide digraph that we are going to construct. Note that every such partition class $U \in \mathcal{U}$ has density precisely $\prod_{i\in [r]}\frac{1}{\ell_i}$ in $\N$.  

For $U,U'\in \mathcal{U}$, either there is an infinite matching of $[U,U']$-edges in colour $r+1$ or there is a finite set of vertices whose deletion leaves no $[U,U']$-edges of colour $r+1$ (by K\"onig's theorem).  In the latter case, delete such a finite set. Let $K$ be the complete symmetric digraph with vertex set $\mathcal{U}$ and let $G$ be the digraph with vertex set $\mathcal{U}$ and edge set $$E(G):= \{(U,U')\colon \text{every $[U,U']$-edge has colour in $[r]$}\}.$$
If there is a cycle in $K-E(G)$, this implies that we can construct a path of colour $r+1$ of density larger than $\prod_{i\in [r]}\frac{1}{\ell_i}$ in $\N$; so suppose not. By Lemma~\ref{Transitive Tournament Lemma}, it follows that $G$ has a spanning transitive tournament $T$.

Let $\vec{T}$ be the lift of $T$ to $\K{\N}$, i.e. $\vec{T}$ is the spanning subgraph of $\K{\N}$ which contains all edges $(m,n)$ where $m\in U$, $n\in U'$ and $(U,U')$ is an auxiliary edge in $E(T)$. Then $\vec{T}$ is also acyclic and spanning. Following the proof of Theorem \ref{Louis Thm} (but applied to $\vec{T}$ instead of all of $\K{\N}$), we define sets $W_{i_1,\dots, i_j}$ for all $j\in [r]$ and $0\leq i_j<\ell_j$ by recursively defining $W_{i_1,\dots, i_m}$ to consist of all vertices $w \in W_{i_1,\dots, i_{m-1}}$ such that the longest path in $\vec{T}[W_{i_1,\dots, i_{m-1}}]$ in colour $m$ starting at $w$ has length~$i_m$.

\begin{claim}{1}
The family $\{W_{i_1,\dots, i_j,0},\dots, W_{i_1,\dots, i_j,\ell_{j+1}-1}\}$ is a partition of $W_{i_1,\dots, i_j}$. 
\end{claim}
\begin{proof}[Proof of Claim~1]\renewcommand{\qedsymbol}{$\Diamond$}
Clear, because there is no path of length $\ell_{m+1}$ in colour $m+1$.
\end{proof}

\begin{claim}{2}
If $w,w'\in W_{i_1,\dots, i_j}$ and $(w,w')\in E(\vec{T})$, then $c(w,w')\notin [j]$. 
\end{claim}
\begin{proof}[Proof of Claim~2]\renewcommand{\qedsymbol}{$\Diamond$}
Let $i_{j}<\ell_{j}$ and suppose for a contradiction that $w,w'\in W_{i_1,\dots,i_{j}}$, $(w,w')\in E(\vec{T})$ and $c(w,w')\in [j]$. Let us write $k:=c(w,w')$. By Claim~1, we have $w,w'\in W_{i_1,\dots,i_k}$. Take a longest path $P$ in $\vec{T}[W_{i_1,\dots, i_{k}}]$ in colour $k$ starting in $w'$. Since $\vec{T}$ is acyclic, we know that $wP$ is a path in $\vec{T}[W_{i_1,\dots, i_{k}}]$ in colour $k$ starting in $w$ contradicting that $w$ is contained in $W_{i_1,\dots, i_{k}}$. 
\end{proof}

So far, we have defined the partition $\mathcal{W}:=\{W_{i_1,\dots, i_r}\colon 0\le i_j<\ell_j \text{ for all } j\in[r]\}$ in terms of the acyclic spanning graph $\vec{T}$. It turns out, however, that we still arrive at the same partition as earlier, when we did this construction with respect to all of $\K{\N}$.

\begin{claim}{3}
We have $\mathcal{U}=\mathcal{W}$.
\end{claim}
\begin{proof}[Proof of Claim~3]\renewcommand{\qedsymbol}{$\Diamond$}
Each $W\in \mathcal{W}$ meets at most one partition class of $\mathcal{U}$. Indeed, suppose for a contradiction that $u\in U\cap W$ and $u'\in U'\cap W$ for different partition classes $U,U'\in \mathcal{U}$ and some $W\in \mathcal{W}$. Since $u\in U$ and $u'\in U'$ we know that either $(u,u')$ or $(u',u)$ is contained in $E(\vec{T})$, say $(u,u')$. Then $(u,u')$ has a colour in $[r]$, in violation of Claim 2. But then it follows from the fact that $\mathcal{W}$ is a partition with $|\mathcal{W}|\le |\mathcal{U}|$ that $\mathcal{U} = \mathcal{W}$.
\end{proof}

Consider the slide graph with regard to the partition $\mathcal{W}$. The following claim shows that the first type of edges in the definition of the slide graph have the right colour:

\begin{claim}{4}
Let $w\in W_{i_1,\dots, i_r}$ and $w'\in W_{j_1,\dots, j_r}$ where for some $k$ we have $i_k<j_k$ and $i_{k'}=j_{k'}$ for $k'\neq k$. Then $(w',w)$ is an edge of $\vec{T}$ and has colour $k$.
\end{claim}
\begin{proof}[Proof of Claim~4]\renewcommand{\qedsymbol}{$\Diamond$}
The claim is proved via induction on $r-k$. For the base case let $r-k=0$, i.e., $k=r$. 
By Claim~3 and the construction of $T$, either $(w,w')$ or $(w',w)$ is contained in $E(\vec{T})$. Suppose, for a contradiction, that $(w,w')\in E(\vec{T})$. 

By definition of $\mathcal{W}$, we find a path $P$ in $\vec{T}[W_{i_1,\dots,i_{r-1}}]$ of length $j_r$ in colour $r$ starting at $w'$. However, since $\vec{T}$ is acyclic, we see that $wP$ is a path in $\vec{T}[W_{i_1,\dots,i_{r-1}}]$ of length $j_r+1$ in colour $r$ starting at $w$, implying that $i_r \geq j_r+1$, contradicting our assumption. 

Now, assume inductively that that $r-k \ge 1$ and Claim~4 holds for integers less than $r-k$. Let $w\in W_{i_1,\dots, i_{r}}$ and $w'\in W_{j_1,\dots, j_r}$. Furthermore, assume that $i_k<j_k$ and $i_{k'}=j_{k'}$ if $k'\neq k$. Again, either $(w,w')$ or $(w',w)$ is contained in $E(\vec{T})$. Let $\vec{e}$ be the unique edge in $E(\vec{T})\cap \{(w,w'),(w',w)\}$. We first show that $\vec{e}$ has colour $k$. 

By Claim~2, we know the colour of $\vec{e}$ is at least $k$. Suppose for a contradiction that $\vec{e}$ has a colour $k'$ at least $k+1$.
\begin{itemize}
\item If $\vec{e}=(w,w')$, fix, for all $i> i_{k'}$, vertices $w_{i}\in W_{i_1,\dots,i_{k'-1} ,i,i_{k'+1},\dots , i_r}$ and, for all $j<j_{k'}$, vertices $w_{j}\in W_{j_1,\dots,j_{k'-1} ,j,j_{k'+1},\dots , j_r}$. Then define $Q$ as the path $$Q:=w_{\ell_{k'}-1}\dots w_{i_{k'}+1}\vec{e}w_{i_{k'}-1}\dots w_2w_1w_0.$$ (cf. Figure \ref{The path $Q$}.)
\item If $\vec{e}=(w',w)$, fix, for all $j> j_{k'}$, vertices $w_{j}\in W_{j_1,\dots,j_{k'-1} ,j,j_{k'+1},\dots , j_r}$ and, for all $i<i_{k'}$, vertices $w_{i}\in W_{i_1,\dots,i_{k'-1} ,i,i_{k'+1},\dots , i_r}$. Then define $Q$ as the path $$Q:=w_{\ell_{k'}-1}\dots w_{i_{k'}+1}\vec{e}w_{i_{k'}-1}\dots w_2w_1w_0.$$  
\end{itemize} 

\begin{figure}
\centering
\begin{tikzpicture}
\draw (-3,1) ellipse (1cm and 0.5cm);
\draw (0,1) ellipse (1cm and 0.5cm);
\draw (3,1) ellipse (1cm and 0.5cm);
\draw (6,1) ellipse (1cm and 0.5cm);

\draw (-3,-1) ellipse (1cm and 0.5cm);
\draw (0,-1) ellipse (1cm and 0.5cm);
\draw (3,-1) ellipse (1cm and 0.5cm);
\draw (6,-1) ellipse (1cm and 0.5cm);

\draw[-{Latex[length=2mm,width=2mm]},green,thick] (6,1) to (3,1);
\draw[-{Latex[length=2mm,width=2mm]},green,thick] (3,1) to (3,-1);
\draw[-{Latex[length=2mm,width=2mm]},green,thick] (3,-1) to (0,-1);
\draw[-{Latex[length=2mm,width=2mm]},green,thick] (0,-1) to (-3,-1);

\node[right,black] at (3.1,0.2) {$W_{i_1\dots ,i_k,\dots,2,\dots,i_r}$};
\node[right,black] at (6.1,0.2) {$W_{i_1\dots ,i_k,\dots,3,\dots,i_r}$};

\node[right,black] at (-2.9,-1.8) {$W_{i_1\dots ,j_k,\dots,0,\dots,i_r}$};
\node[right,black] at (-0.1,-1.8) {$W_{i_1\dots ,j_k,\dots,1,\dots,i_r}$};

\node[right,black] at (2.55,0) {$\vec{e}$}; 

\end{tikzpicture}
\caption{The path $Q$ (green) for $\ell_{k'}=4$}\label{The path $Q$}
\end{figure}
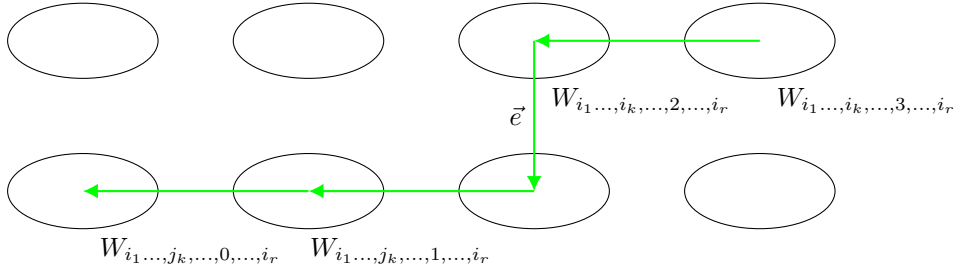
In either case, it follows from our inductive hypothesis that $Q$ is a monochromatic path in colour $k'$. However, $Q$ now has length $\ell_{k'}$, a contradiction.

Finally, to show that $\vec{e}=(w',w) \in E(\vec{T})$, we may now argue as in the base case: By construction of $\mathcal{W}$, find a longest path $P$ starting at $w'$ in colour $k$ of length $j_k$ in $\vec{T}[W_{i_1,\dots, i_{k-1}}]$. If $\vec{e}=(w,w')$, then $\vec{e}P$ is a path starting at $w$ in colour $k$ in $\vec{T}[W_{i_1,\dots, i_{k-1}}]$ of length $j_k+1$, forcing that $i_k \geq j_k + 1$, a contradiction.
\end{proof}

The following claim completes the proof:

\begin{claim}{5}
Let $w\in W_{i_1,\dots,i_r}$ and $w'\in W_{j_1,\dots,j_r}$ with $i_k\le j_k$ for every $k\in {r}$. Then $(w,w')$ has colour $r+1$.  
\end{claim}
\begin{proof}[Proof of Claim~5]\renewcommand{\qedsymbol}{$\Diamond$}
Suppose for a contradiction, that $(w,w')$ has a colour $k'\in [r]$. Fix, for all $i> i_{k'}$, vertices $w_{i}\in W_{i_1,\dots,i_{k'-1} ,i,i_{k'+1},\dots , i_r}\backslash\{w,w'\}$ and, for all $j<j_{k'}$, vertices $w_{j}\in W_{j_1,\dots,j_{k'-1} ,j,j_{k'+1},\dots , j_r} \backslash \{w,w'\}$ (all distinct). By Claim 4, the path $$w_{\ell_{k'}-1}\dots w_{i_{k'}+1}\vec{e}w_{i_{k'}-1}\dots w_2w_1w_0$$ has colour $k'$ and since $i_k\le j_k$ it has length at least $\ell_{k'}$ (contradiction).
\end{proof}
With this final claim established, the proof is complete.
\end{proof}

Since the slide graph is just the complete symmetric digraph on $\N$ for $r=1$, we obtain the stability result of Guggiari in \cite[Theorem~1.4]{G17} as a corollary: 

\begin{cor}[{cf. \cite[Theorem~1.4]{G17}}]
Take any $2$-colouring of $\K{\N}$ in which there are no monochromatic directed paths of length $\ell$ in colour $1$ and every monochromatic directed path in colour $2$ has upper density at most $\frac{1}{\ell}$. Then there exists a finite set of vertices $U$ such that the $2$-colouring induced on $\N\backslash U$ is isomorphic to the cube colouring on $\K{\N}$.
\end{cor}

Unless $r=1$, there are still edges whose colours we did not specify---those which are not part of the slide graph. In fact, the colours of these edges can vary depending on the choice of the colouring:

\begin{example}
Let $r\ge 2$ and $\ell_1,\dots,\ell_r$ positive integers. Consider the cube colouring $c_\U$ and the slide graph $D_\U$ on a cube partition $\U$ of order $(\ell_1,\dots, \ell_r)$ with equally (upper-) dense partition classes. Fix $U_{i_1,\dots, i_r}$ and $U_{j_1,\dots ,j_r}$ such that the $[U_{i_1,\dots, i_r}, U_{j_1,\dots j_r}]$ edges are not part of $D_\U$. Colour the $[U_{i_1,\dots, i_r}, U_{j_1,\dots ,j_r}]$-edges with arbitrary colours from $\{k\in [r]\colon i_k>j_k\}$ and all the other edges as in the cube colouring. Almost the same proof as in Theorem \ref{thm_cubecolouringworks} shows that there is no directed path of length $\ell_i$ in colour $i$ for any $i\in [r]$ and that every directed monochromatic path of colour $r+1$ has upper density at most $\prod_{i\in [r]}\frac{1}{\ell_i}$.
\end{example}

However, we observe (using the terminology of the proof of Theorem~\ref{Stability Thm}):
\begin{itemize}
\item If $i_k\le j_k$, then no $[W_{i_1,\dots,i_r}, W_{j_1,\dots,j_r}]$-edge has colour $k$ (cf. the proof of Claim 5).
\item If $i_k>j_k$ for every $k\in [r]$, then there is a finite set of vertices $F$ such that all $[(W_{j_1,\dots,j_r}\backslash F), (W_{i_1,\dots,i_r}\backslash F)]$-edges have a colour in $[r]$. 
\end{itemize}

\section{Acknowledgments}

The authors thank Paul McKenney for some productive discussions.


\begin{thebibliography}{99}
\bibitem{BP17} C.~B\"urger, M.~Pitz, Decomposing edge-coloured complete symmetric digraphs into monochromatic paths, manuscript, {\em http://arxiv.org/abs/1711.08711} (2017).
\bibitem{C17} J.~Corsten, {\em Personal communication} (2017).
\bibitem{DM16} L.~DeBiasio and P.~McKenney, Density of monochromatic infinite subgraphs, {\em http://arxiv.org/abs.1611.05423} (2016).
\bibitem{ESSS17} M.~Elekes, D.~Soukup, L.~Soukup and Z.~Szentmikl\`{o}ssy, Decompositions of edge-coloured infinite complete graphs into monochromatic paths, {\em Discrete Mathematics} {\bf 340} (2017), 2053--2069.
\bibitem{EG93} P.~Erd\H{o}s and F.~Galvin, Monochromatic infinite paths, {\em Discrete Mathematics} {\bf 113} (1993), 59--70.
\bibitem{G68} T.~Gallai, On directed paths and circuits, \emph{Theory of graphs} (1968), 115--118.
\bibitem{GG} L.~Gerencs\'er and A.~Gy\'arf\'as.  On Ramsey-type problems, \emph{Ann. Sci. Budapest. E\"otv\"os Sect. Math}, \textbf{10} (1967), 167--170.  
\bibitem{G17} H.~Guggiari, Monochromatic paths in the complete symmetric infinite digraph, manuscript, {\em http://arxiv.org/abs.1710.10900} (2017).
\bibitem{H65} M.~Hasse, Zur algebraischen Begr{\"u}ndung der Graphentheorie. I, \emph{Mathematische Nachrichten} \textbf{28} (1965), no.~5-6, 275--290.
\bibitem{R78} R.~Rado, Monochromatic paths in graphs, {\em Ann Discrete Math.} {\bf 3} (1978), 191--194.
\bibitem{R73} H.~Raynaud, Sur le circuit hamiltonien bi-colour\'{e} dans les graphes orient\'{e}s, {\em Periodica Mathematica Hungarica} {\bf 3} (1973), 289--297.
\bibitem{R67} B.~Roy, Nombre chromatique et plus longs chemins d'un graphe, \emph{Revue fran{\c{c}}aise d'informatique et de recherche op{\'e}rationnelle} \textbf{1} (1967), no.~5, 129--132.
\bibitem{S17} D.~Soukup, Decompositions of edge-colored infinite complete graphs into monochromatic paths II \emph{Israel Journal of Mathematics} {\bf 221} (2017), 235-273.
\bibitem{V62} L.~M.~Vitaver, Determination of minimal colouring of vertices of a graph by means of boolean powers of the incidence matrix, \emph{Dokl. Akad. Nauk SSSR}, \textbf{147} (1962), p.~728.
\end{thebibliography}
\end{document}